\documentclass[12pt,a4paper]{article}
\usepackage{amssymb}
\usepackage{amsfonts}
\usepackage{amsmath}
\usepackage{graphicx}

\newtheorem{theorem}{Theorem}[section]

\newtheorem{condition}{Condition}

\newtheorem{corollary}[theorem]{Corollary}

\newtheorem{definition}{Definition}[section]

\newtheorem{lemma}[theorem]{Lemma}
\newtheorem{notation}{Notation}

\newtheorem{proposition}[theorem]{Proposition}
\newtheorem{remark}{Remark}[section]

\newenvironment{proof}[1][Proof]{\textbf{#1.} }{\ \rule{0.5em}{0.5em}}

\begin{document}

\title{Random walks on graphs with volume and time doubling\thanks{%
Running head: Random walks on graphs}}
\author{Andr\'{a}s Telcs \\
{\small Department of Computer Science and Information Theory, }\\
{\small Budapest} {\small University of Technology and Economics}\\
{\small telcs@szit.bme.hu}}
\maketitle

\begin{abstract}
This paper studies the on- and off-diagonal upper estimate and the two-sided
transition probability estimate of random walks on weighted graphs.

\ \ \ \ \ MSC2000 60J10, 60J45, 35B05

\end{abstract}

\tableofcontents

\section{Introduction}

\setcounter{equation}{0}\label{sintr}

Let us consider a countable infinite connected graph $\Gamma $. \ A weight
function $\mu _{x,y}=\mu _{y,x}>0$ is given on the edges $x\sim y.$ This
weight induces a measure $\mu (x)$%
\begin{equation*}
\mu (x)=\sum_{y\sim x}\mu _{x,y},\text{ }\mu (A)=\sum_{y\in A}\mu (y)
\end{equation*}%
on the vertex set $A\subset \Gamma $ and defines a reversible Markov chain $%
X_{n}\in \Gamma $, i.e. a random walk on the weighted graph $(\Gamma ,\mu )$
with transition probabilities
\begin{align*}
P(x,y)& =\frac{\mu _{x,y}}{\mu (x)}, \\
P_{n}(x,y)& =\mathbb{P}(X_{n}=y|X_{0}=x).
\end{align*}%
For a set $A\subset \Gamma $ the killed random walk is defined by the
transition operator restricted to $c_{0}\left( A\right) ,$ and the
corresponding transition probability is denoted by $P_{n}^{A}\left(
x,y\right) $

The graph is equipped with the usual (shortest path length) graph distance $%
d(x,y)$ and open metric balls are defined for $x\in \Gamma ,$ $R>0$ as $%
B(x,R)=\{y\in \Gamma :d(x,y)<R\}$ and its $\mu -$measure is denoted by
\begin{equation}
V(x,R)=\mu \left( B\left( x,R\right) \right) .  \label{vdef}
\end{equation}%
If $\Gamma =%
\mathbb{Z}
^{d}$ and $\mu _{x,y}=1$ if $d\left( x,y\right) =1$ we get back the
classical, simple symmetric nearest neighbor random walk on $%
\mathbb{Z}
^{d}$. This random walk serves as a discrete approximation and model for the
diffusion in continuous space and time. It is widely accepted that the
interesting phenomena and results found on continuous space and time have
their random walks counterparts ( and vice versa) (c.f. just as example \cite%
{KS},\cite{BBK} of the link between the two frameworks). \ The first
rigorously studied fractal type graph was the Sierpinski triangle.   On this graph the volume growth is
polynomial%
\begin{equation*}
V\left( x,R\right) \simeq R^{\alpha }
\end{equation*}%
with exponent $\alpha =\frac{\log 3}{\log 2}$. Here $\simeq $ means that the
ratio of the two functions of $r$ is uniformly separated from zero and
infinity. \ On this infinite graph the transition probability estimate has
the form ( c.f. \cite{J})
\begin{eqnarray*}
p_{n}\left( x,y\right) &\leq &\frac{C}{n^{\alpha /\beta }}\exp \left(
-c\left( \frac{d^{\beta }\left( x,y\right) }{n}\right) ^{\frac{1}{\beta -1}%
}\right) , \\
p_{n}\left( x,y\right) +p_{n+1}\left( x,y\right) &\geq &\frac{c}{n^{\alpha
/\beta }}\exp \left( -C\left( \frac{d^{\beta }\left( x,y\right) }{n}\right)
^{\frac{1}{\beta -1}}\right) ,
\end{eqnarray*}%
where $C,c>0,$ and the walk dimension is $\beta =\frac{\log 5}{\log 2}>2.$
This walk moves slower than the classical one due to the big holes and
narrow connections. \ This is reflected in the exponent $\beta >2$.\ In the
classical $%
\mathbb{Z}
^{d}$ case the mean exit time $E\left( x,R\right) \simeq R^{2},$ which is
the expected value of the time needed by the walk to leave the ball $B\left(
x,R\right) .$ For the Sierpinski graph it is $E\left( x,R\right) \simeq
R^{\beta }$ with $\beta =\frac{\log 5}{\log 2}>2.$ Many efforts have been
devoted to the investigation of other particular fractals and general
understanding what kind of structural properties are responsible for the
leading and exponential term of the upper and lower estimate (for further
background and literature please see \cite{B}, \cite{HK}). The next
challenge was to obtain such kind of "heat kernel" estimates \ on a larger
class of graphs and drop the Alforth regularity: $V\left( x,R\right) \simeq
R^{\alpha }$. \ An easy example for such a graph is described in \cite{GT2}.
\ The Vicsek tree is considered, which is built in a recursive way. If the
weights assigned to the edges are slightly increasing by the distance to the
root, the resulted weighted graph is not Alforth regular any more, ( \ref{wvicsek}) but satisfies the volume doubling condition (see Definition \ref%
{dvd}.) In \cite{GT2} the authors gave necessary and sufficient conditions
for two-sided sub-Gaussian estimates of the following form,%
\begin{eqnarray}
p_{n}\left( x,y\right) &\leq &\frac{C}{V\left( x,n^{1/\beta }\right) }\exp
\left( -c\left( \frac{d^{\beta }}{n}\right) ^{\frac{1}{\beta -1}}\right) ,
\label{ueb} \\
p_{n}\left( x,y\right) +p_{n+1}\left( x,y\right) &\geq &\frac{c}{V\left(
x,n\alpha ^{1/\beta }\right) }\exp \left( -C\left( \frac{d^{\beta }}{n}%
\right) ^{\frac{1}{\beta -1}}\right)  \label{leb}
\end{eqnarray}%
which is local in the volume $V\left( x,R\right) $ but the mean exit time is
uniform with respect to the space, $E\left( x,R\right) \simeq R^{\beta }\,$.
(See \cite{B},\cite{GT1},\cite{GT2} or \cite{TD} for further remarks and
history of the the heat kernel estimates.) One can rise the next natural
question. \

\begin{quote}
What can be said if the mean exit time is not polynomial, and what if it
depends on the center of the test ball?
\end{quote}

The present paper answers both questions. Before we explain the results let
us see an example based again on the Vicsek tree. Now the edges replaced
with paths of slowly increasing length (as we depart from the root) . Let us consider a vertex $x$ in a middle of a
distant path. Let be the radius $r$ of the test
ball is smaller (or comparable) to the half of the length of the path, then
we have a classical one dimensional walk in that ball, consequently $E\left(
x,r\right) \simeq r^{2}$. While for big $R$ we have the large scale behavior
of the Vicsek tree, hence $E\left( x,R\right) \simeq R^{\frac{\log 15}{\log 3%
}}>>R^{2}$ for large $R.$ \ On the other hand all balls centered at the root
has the usual behavior of the Vicsek tree $E\left( x,R\right) \simeq R^{%
\frac{\log 15}{\log 3}}.$ \ We shall see that this stretched Vicsek tree has
all the properties needed to obtain an upper bound for the heat kernel. \
The details of this example will be given in Section \ref{sexample}. Several
further graphs can be constructed in a similar way. \ For instance we
consider a graph which possesses some nice properties and replace the edges
(or well-defined sub-graphs) with elements of a class of graphs (again with
increasing size as we depart from a reference vertex) connecting them on a
subset of prescribed vertices. Here we replace the edges of the
Vicsek tree with diamonds formed by two Sierpinski triangles \ As the
distance grows from the root, bigger Sierpinski triangles are inserted. (To
keep the needed properties of the graph the increase of the size of the
triangles should be slow.)

In order to present the main results we have to introduce the essential
notions of the paper.

\begin{definition}
\label{dvd}The weighted graph has the \emph{volume doubling }$\left( \mathbf{%
VD}\right) $ property if there is a constant $D_{V}>0$ such that for all $%
x\in \Gamma $ and $R>0$%
\begin{equation}
V(x,2R)\leq D_{V}V(x,R).  \label{PD1V}
\end{equation}
\end{definition}

\begin{notation}
For convenience we introduce a short notation for the volume of the annulus;
$v=v(x,r,R)=V(x,R)-V(x,r).$
\end{notation}

\begin{notation}
For two real sequences $a_{\xi },b_{x}$ we will write%
\begin{equation*}
a_{\xi }\simeq b_{\xi }
\end{equation*}%
if there is a $C\geq 1$ such that for all $\xi $%
\begin{equation*}
\frac{1}{C}a_{\xi }\leq b_{\xi }\leq Ca_{\xi }.
\end{equation*}
\end{notation}

Now let us consider the exit time
\begin{equation*}
T_{B(x,R)}=\min \{k:X_{k}\notin B(x,R)\}
\end{equation*}%
from the ball $B(x,R)$ and its mean value
\begin{equation*}
E_{z}(x,R)=\mathbb{E}(T_{B(x,R)}|X_{0}=z)
\end{equation*}
and let us use the
\begin{equation*}
E(x,R)=E_{x}(x,R)
\end{equation*}
short notation.

\begin{definition}
We will say that the weighted graph $(\Gamma ,\mu )$ satisfies the time
comparison principle $\left( \mathbf{TC}\right) $ if there is a constant $%
C_{T}>1$ such that for all $x\in \Gamma $ and $R>0,y\in B\left( x,R\right) $%
\begin{equation}
\frac{E(x,2R)}{E\left( y,R\right) }\leq C_{T}.  \label{TC}
\end{equation}
\end{definition}

\begin{definition}
We will say that $(\Gamma ,\mu )$ has \emph{time doubling} property $\left(
\mathbf{TD}\right) $ if there is a $D_{T}>0$ such that for all $x\in \Gamma $
and $R\geq 0$%
\begin{equation}
E(x,2R)\leq D_{T}E(x,R).  \label{TD}
\end{equation}
\end{definition}

\begin{remark}
It is clear that from the time doubling property $\left( TD\right) $ it
follows that there are $C>0$ and $\beta >0$ such that for all $x\in \Gamma $
\ and $R>r>0$%
\begin{equation}
\frac{E(x,R)}{E(x,r)}\leq C\left( \frac{R}{r}\right) ^{\beta }.
\label{tdbeta}
\end{equation}
\end{remark}

Basically the volume doubling condition $\left( \ref{PD1V}\right) $ and the
time comparison principle $\left( \ref{TC}\right) $ specify the local
framework for our study. It is clear that $\left( TC\right) $ implies $%
\left( TD\right) ,$ the time doubling property while the reverse seems to be
not true even if $\left( VD\right) $ is assumed.

The goal of the present paper is twofold. \ First we would like to give
characterization of graphs which have on- and off-diagonal upper estimate if
neither the volume nor the mean exit time is uniform in the space like in
the above examples. \ Secondly we give characterization of graphs which have
two-sided heat kernel estimates. \ For that we consider graphs with the
volume doubling property and it is assumed that the mean exit time is
uniform in the space, more precisely satisfies $\left( \mathbf{E}\right) $:%
\begin{equation}
E(x,R)\simeq E\left( y,R\right)  \label{EF}
\end{equation}%
holds, i.e. the mean exit time does not depend on the center of the ball.
The \textit{semi-local framework} will be the determined by the conditions $%
\left( VD\right) $ and $\left( E\right) .$

In \cite{TD}, it was shown that for strongly recurrent graphs upper
estimates can be obtained in the local framework and two-sided estimates in
the semi-local framework. \ Here we present similar result dropping the
condition \ of strong recurrence and generalize them in many respect.

\begin{condition}
In several statements we assume that condition $\mathbf{(p}_{0}\mathbf{)}$
holds, that is, there is a universal $p_{0}>0$ such that for all $x,y\in
\Gamma ,x\sim y$
\begin{equation}
\frac{\mu _{x,y}}{\mu (x)}\geq p_{0}.  \label{p0}
\end{equation}
\end{condition}

\begin{notation}
For a set $A\subset \Gamma $ denote the closure by
\begin{equation*}
\overline{A}=\left\{ y\in \Gamma :\text{there is an }x\in A\text{ such that }%
x\sim y\right\} .
\end{equation*}
The external boundary is defined as $\partial A=\overline{A}\backslash A.$
\end{notation}

\begin{definition}
A function $h$ is harmonic on a set $A\subset \Gamma $ if it is defined on $%
\overline{A}$ and
\begin{equation*}
Ph\left( x\right) =\sum_{y}P\left( x,y\right) h\left( y\right) =h\left(
x\right)
\end{equation*}%
for all $x\in A.$
\end{definition}

\

\begin{theorem}
\label{tLDUE}For any weighted graph $(\Gamma ,\mu )$ if $\left( p_{0}\right)
,(VD)$ and\ $\left( TC\right) $ hold, then the following statements are
equivalent:

\begin{enumerate}
\item The mean value inequality $\left( \mathbf{MV}\right) $ holds: there is
a $C>0$ such that for all $x\in \Gamma ,R>0$ and for all $u\geq 0$ harmonic
functions on $B=B\left( x,R\right) $%
\begin{equation}
u\left( x\right) \leq \frac{C}{V\left( x,R\right) }\sum_{y\in B}u\left(
y\right) \mu \left( y\right) ,  \label{MV}
\end{equation}

\item the local diagonal upper estimate $\left( \mathbf{DUE}\right) $ holds:
there is a $C>0$ such that for all $x\in \Gamma ,n>0$
\begin{equation}
p_{n}(x,x)\leq \frac{C}{V(x,e(x,n))},  \label{LDUE}
\end{equation}%
where $e\left( x,n\right) $ is the inverse of $E\left( x,R\right) $ in the
second variable.

\item the upper estimate $\left( \mathbf{UE}\right) $ holds: there are $%
C,\beta >1,c>0$ such that for all $x,y\in \Gamma ,$ $n>0$%
\begin{equation}
p_{n}\left( x,y\right) \leq \frac{C}{V\left( x,e\left( x,n\right) \right) }%
\exp \left[ -c\left( \frac{E\left( x,d\left( x,y\right) \right) }{n}\right)
^{\frac{1}{\beta -1}}\right] .  \label{UE1}
\end{equation}
\end{enumerate}
\end{theorem}

The existence of $e$ will be clear from the properties of the mean exit time
(c.f. Section \ref{sloc}).

This theorem can be given in a different form if we introduce the skewed
version of the parabolic mean value inequality.

\begin{definition}
\label{dspmv}We shall say that the skewed parabolic mean value inequality
holds if for $0<c_{_{1}}<c_{_{2}}$ constants there is a $C\geq 1$ such that
for all $R>0,x\in \Gamma ,y\in B=B\left( x,R\right) $ for all non-negative
Dirichlet solutions $u_{k}$ of the heat equation%
\begin{equation}
P^{B}u_{k}=u_{k+1}  \label{he}
\end{equation}%
on $\left[ 0,c_{2}E\left( x,R\right) \right] \times B\left( x,R\right) $%
\begin{equation}
u_{n}(x)\leq \frac{C}{V(y,2R)E\left( y,2R\right) }\sum_{i=c_{1}E}^{c_{2}E}%
\sum_{z\in B(x,R)}u_{i}(z)\mu (z)\,  \label{sPMV}
\end{equation}%
satisfied, where $E=E\left( x,R\right) ,n=c_{2}E$.
\end{definition}

\begin{remark}
\label{rsPMV}One can see easily with the choice of $u_{i}\left( y\right)
\equiv 1$ that $\left( \ref{sPMV}\right) $ implies $\left( VD\right) $ and $%
\left( TC\right) .$
\end{remark}

Having this condition the above theorem can be restated as follows.

\begin{theorem}
\label{tsPMV} If $\left( \Gamma ,\mu \right) $ satisfies $\left(
p_{0}\right) $ then the following conditions are equivalent:

\begin{enumerate}
\item the skewed parabolic mean value inequality, $\left( \ref{sPMV}\right) $
holds,

\item $\left( VD\right) ,\left( TC\right) $ and $\left( MV\right) $ holds,

\item $\left( VD\right) ,\left( TC\right) $ and the diagonal upper estimate,
$\left( DUE\right) $ holds,

\item $\left( VD\right) ,\left( TC\right) $ and the upper estimate $\left(
UE\right) $ holds.
\end{enumerate}
\end{theorem}

The above results deal with graphs which satisfy the volume doubling
property and time comparison principle. Let us observe that the exponent in $%
\left( \ref{UE1}\right) $ depend on $x$ \ not only on the distance between $%
x $ and $y$. To find matching exponents for the upper and lower off-diagonal
estimates it seems plausible to assume that the mean exit time is (up to a
constant) is uniform in the space, that is it satisfies $\left( E\right) :$%
\begin{equation*}
E(x,R)\simeq E\left( y,R\right) .
\end{equation*}%
It is connivent to specify a function $F(R)$ for $R\geq 0$%
\begin{equation*}
F(R)=\inf_{x\in \Gamma }E(x,R),
\end{equation*}%
for which from $\left( E\right) $ it follows that $\ $there is a $C_{0}>1$
such that for all $x\in \Gamma $ and $R\geq 0$%
\begin{equation}
F(R)\leq E(x,R)\leq C_{0}F(R).  \label{E}
\end{equation}%
This function inherits certain properties of $E(x,R),$ first of all from the
time doubling property it follows that $F$ also has doubling property:
\begin{equation}
F(2R)\leq D_{E}F(R).  \label{ED1F}
\end{equation}

We shall say that $F$ is a (very) \textit{proper space time scale function}
if it has certain properties which will be defined in Section \ref{ssloc}
(c.f. Definition \ref{dpropf}).

The function $F$ with the inherited properties will take over the role of $%
R^{\beta }$ $($or $R^{2}).$ \ \ The inverse function of $F,$ $f(.)=F^{-1}(.)$
takes over the role of $(R^{\frac{1}{\beta }})$ $R^{\frac{1}{2}}$ in the
(sub-) Gaussian estimates.\ The existence of $\ f$ \ will be shown in
Section \ref{ssloc}.

\begin{definition}
The sub-Gaussian kernel function with \ respect to a function $F$ is $%
k=k(n,R)\geq 1$, defined as the maximal integer for which
\begin{equation}
\frac{n}{k}\leq F(\left\lfloor \frac{R}{k}\right\rfloor )  \label{kdef}
\end{equation}%
or $k=0$ by definition if there is no appropriate $k$.
\end{definition}

\begin{definition}
The transition probability satisfies $\left( \mathbf{UE}_{F}\right) ,$ the
sub-Gaussian upper estimate with respect to $F$ if there are $c,C>0$ such
that for all $x,y\in \Gamma ,n>0$%
\begin{equation}
p_{n}(x,y)\leq \frac{C}{V(x,f(n))}\exp -ck(n,d(x,y)),  \label{UEF}
\end{equation}%
and $\left( \mathbf{LE}_{F}\right) ,$ the sub-Gaussian lower estimate is
satisfied if
\begin{equation}
\widetilde{p}_{n}(x,y)\geq \frac{c}{V(x,f(n))}\exp -Ck(n,d(x,y)),
\label{LEF}
\end{equation}%
where $\widetilde{p_{n}}=p_{n}+p_{n+1}.$
\end{definition}

In the semi-local framework we have the following result.

\begin{theorem}
\label{tmain}If a weighted graph $(\Gamma ,\mu )$ satisfies $\left(
p_{0}\right) $ then the following statements are equivalent:

\begin{enumerate}
\item for a very proper $F$ $(UE_{F})$ and $(LE_{F})$ hold

\item for a very proper $F$ the F-parabolic Harnack inequality holds,

\item $(VD),(E)$ and the elliptic Harnack inequality hold.
\end{enumerate}
\end{theorem}

The definition of the elliptic and parabolic Harnack inequality and some
other definitions are given in Sections \ref{sprel} and \ref{ssloc}.

In Section \ref{ssloc} characterization of graphs satisfying separately the
upper estimate $\left( UE_{F}\right) $ will also be given

Let us mention that in this generality Hebisch and Saloff-Coste in \cite{HS}
proved the equivalence of $1$. and $2$ of Theorem \ref{tmain}.

The complete characterization of graphs which have two-sided heat kernel
Gaussian estimates (c.f. $\left( \ref{ueb}\right) $, $\left( \ref{leb}%
\right) $ with $\beta =2$ ) was given by Delmotte \cite{De}. This
characterization has been extended to two-sided sub-Gaussian estimates ( $%
\beta \geq 2$ ) in \cite{GT2}.

In a recent work Li and Wang \cite{LW} proved in the context of complete
Riemannian manifolds that if the \textit{volume doubling property }holds%
\textit{\ }then the following particular upper bound for the Green kernel, $%
g^{B}\left( x,y\right) =\int_{0}^{\infty }p_{t}^{B}\left( x,y\right) dt,$
for $B=B\left( x,R\right) $
\begin{equation}
g^{B}\left( p,x\right) \leq C\int_{d^{2}\left( x,p\right) }^{\left(
CR\right) ^{2}}\frac{dt}{V\left( p,\sqrt{t}\right) },  \label{g1}
\end{equation}%
implies the mean value inequality
\begin{equation}
u\left( x\right) \leq \frac{C}{V\left( x,R\right) }\int_{B\left( x,R\right)
}u\left( y\right) dy  \label{contMV}
\end{equation}%
for $u\geq 0$ sub-harmonic functions. \ The opposite direction remained open
question. Here we show that for weighted graphs in the semi-local framework,
namely under conditions $\left( VD\right) ,\left( TD\right) $ and \ $\left(
E\right) ,$ the mean value inequality $(MV)$ implies a Green's function
bound equivalent\ to $\left( \ref{g1}\right) $.

The organization of the paper is the following. \ In Section \ref{sprel} we
collect the basic definitions and preliminary observations. \ In Section \ref%
{sloc} we discuss the local framework and present a much more detailed
version of Theorem \ref{tLDUE}. In Section \ref{ssloc} we study the
semi-local setup and prove Theorem \ref{tmain}.

\textit{Acknowledgment}

\textit{The author is deeply indebted to Professor Alexander Grigor'yan. \
Neither this nor recent other papers of the author would exists without his
ideas and friendly support. Particularly he proposed to study what is the
necessary and sufficient condition of the off-diagonal upper estimate.}

\section{Basic definitions and preliminaries}

\setcounter{equation}{0}\label{sprel} In this section we recall basic
definitions and observations (mainly from \cite{TD} but we warn the reader
that there are minor deviations from the conventions have been used there).

\begin{definition}
The random walk defined on $(\Gamma ,\mu )$ will be denoted by $(X_{n}).$ It
is a reversible Markov chain on the state space $\Gamma $, reversible with
respect to the measure $\mu $ and has one step transition probability
\begin{equation*}
\mathbb{P(}X_{n+1}=y|X_{n}=x)=P(x,y)=\frac{\mu _{x,y}}{\mu (x)}.
\end{equation*}
\end{definition}

\subsection{Volume doubling}

The volume function $V$ has been already defined in $\left( \ref{vdef}%
\right) $.

\begin{remark}
It is evident that on weighted graphs the volume doubling property $\left(
VD\right) $ is equivalent with the \textit{volume comparison principle},
namely there is a constant $C_{V}>1$ such that for all $x\in \Gamma $ and $%
R>0,y\in B\left( x,R\right) $%
\begin{equation}
\frac{V(x,2R)}{V\left( y,R\right) }\leq C_{V}.  \label{VC}
\end{equation}
\end{remark}

\begin{proposition}
\label{plocvol}If $\left( p_{0}\right) $ holds, then, for all $x,y\in \Gamma
$ and $R>0$ and for some $C>1$,
\begin{equation}
V(x,R)\leq C^{R}\mu (x),  \label{vbound}
\end{equation}%
\begin{equation}
p_{0}^{d(x,y)}\mu (y)\leq \mu (x),
\end{equation}%
and for any $x\in \Gamma $
\begin{equation}
\left\vert \left\{ y:y\sim x\right\} \right\vert \leq \frac{1}{p_{0}}.
\label{deg}
\end{equation}
\end{proposition}

\begin{proof}
\bigskip (c.f. \cite[Proposition 3.1]{GT1})
\end{proof}

\begin{remark}
It follows from the inequality (\ref{vbound}) that, for a fixed $R_{0}$, for
all $R<R_{0}$ $V(x,R)\simeq \mu (x)$.
\end{remark}

\begin{remark}
It is easy to show (c.f. \cite{cg}) that the volume doubling property
implies an anti-doubling property: there is an $A_{V}>1$ such that for all $%
x\in \Gamma ,R>0$%
\begin{equation}
2V(x,R)\leq V(x,A_{V}R).  \label{aVD}
\end{equation}%
One can also show that $\left( VD\right) $ is equivalent with
\begin{equation*}
\frac{V(x,R)}{V(y,S)}\leq C\left( \frac{R}{S}\right) ^{\alpha },
\end{equation*}%
where $\alpha =\log _{2}D_{V}$ and $d(x,y)\leq R$, which is the original
form of Gromov's volume comparison inequality (c.f. \cite{GM}). (For the
proof see again$_{{}}$\cite{cg})
\end{remark}

\begin{remark}
An other direct consequence of $\left( p_{0}\right) $ and $(VD)$ is that
\begin{equation}
v\left( x,R,2R\right) =V(x,2R)-V(x,R)\simeq V(x,R)  \label{V3}
\end{equation}
\end{remark}

\subsection{The resistance}

\begin{definition}
For any two disjoint sets, $A,B\subset \Gamma ,$ the resistance, $\rho
(A,B), $ is defined as
\begin{equation*}
\rho (A,B)=\left( \inf \left\{ \left( \left( I-P\right) f,f\right) _{\mu
}:f|_{A}=1,f|_{B}=0\right\} \right) ^{-1}
\end{equation*}%
and we introduce
\begin{equation*}
\rho (x,S,R)=\rho (B(x,S),\Gamma \backslash B(x,R))
\end{equation*}%
for the resistance of the annulus around $x\in \Gamma ,$ with $R>S\geq 0$.
\end{definition}

\begin{definition}
We say that the product of the resistance and volume of the annulus is
uniform in the space if for all $x,y\in \Gamma ,R\geq 0$
\begin{equation}
\rho (x,R,2R)v(x,R,2R)\simeq \rho (y,R,2R)v(y,R,2R).  \label{rrvA}
\end{equation}
\end{definition}

\begin{corollary}
\label{crav>l2}For all weighted graphs, $x\in \Gamma ,r\geq s\geq 0$%
\begin{equation}
\rho (x,s,r)v(x,s,r)\geq (r-s)^{2},  \label{rv>2}
\end{equation}
\end{corollary}

\begin{proof}
The idea of the proof taken from \cite{T1}, for the details see \cite{ter}.
\end{proof}

\subsection{The mean exit time}

Let us introduce the exit time $T_{A}.$

\begin{definition}
The exit time from a set $A$ is defined as
\begin{equation*}
T_{A}=\min \{k:X_{k}\in \Gamma \backslash A\},
\end{equation*}%
its expected value is denoted by
\begin{equation*}
E_{x}(A)=\mathbb{E}(T_{A}|X_{0}=x)
\end{equation*}%
and let us use the $E=E(x,R)=E_{x}(x,R)$ short notation.
\end{definition}

In this section we introduce several properties of the mean exit time which
will play crucial role in the whole sequel.

The time comparison principle evidently implies the following weaker
inequality: there is a $C>0$ such that
\begin{equation}
\frac{E(x,R)}{E(y,R)}\leq C  \label{pd2e}
\end{equation}%
\ for all $x\in \Gamma ,R>0,y\in B\left( x,R\right) .$ One can observe that $%
\left( \ref{pd2e}\right) $ is the difference between $\left( TC\right) $ and
$\left( TD\right) .$ It is easy to see that $\left( TC\right)
\Longleftrightarrow \left( TD\right) +$ $\left( \ref{pd2e}\right) .$ It also
follows easily that $\left( TC\right) $ is equivalent with the existence of
a $C,\beta \geq 1$ constants for which
\begin{equation}
\frac{E(x,R)}{E(y,S)}\leq C\left( \frac{R}{S}\right) ^{\beta },
\label{tcbeta}
\end{equation}%
for all $y\in B\left( x,R\right) ,R\geq S>0$.

\begin{remark}
\label{re<1/p} It is easy to see that condition $\left( p_{0}\right) $
implies that for all $x\in \Gamma ,R\geq 1$%
\begin{equation*}
E\left( x,R\right) \leq \left( \frac{1}{p_{0}}\right) ^{R}
\end{equation*}
\end{remark}

\begin{definition}
The maximal mean exit time is defined as
\begin{equation*}
\overline{E}(A)=\max_{x\in A}E_{x}(A)
\end{equation*}%
and particularly the $\overline{E}(x,R)=\overline{E}(B(x,R))$ notation will
be used.
\end{definition}

\begin{definition}
The local kernel function $\underline{k},$ $\underline{k}=\underline{k}%
(n,x,R)\geq 1,$ is defined as the maximal integer for which
\begin{equation}
\frac{n}{k}\leq \min\limits_{y\in B(x,R)}E(y,\left\lfloor \frac{R}{k}%
\right\rfloor )  \label{iter}
\end{equation}%
or $\underline{k}=0$ by definition if there is no appropriate $k$.
\end{definition}

\subsection{Mean value inequalities}

\begin{definition}
The random walk on the weighted graph is a reversible Markov chain and the
Markov operator $P$ \ is naturally defined by%
\begin{equation*}
Pf\left( x\right) =\sum P\left( x,y\right) f\left( y\right) .
\end{equation*}
\end{definition}

\begin{definition}
The Laplace operator on the weighted graph $\left( \Gamma ,\mu \right) $ is
defined simply as
\begin{equation*}
\Delta =P-I.
\end{equation*}
\end{definition}

\begin{definition}
For $A\subset \Gamma $ consider the Markov operator $P^{A}$ restricted to $%
A. $ This operator is the Markov operator of the killed Markov chain, which
is killed on leaving $A,$ also corresponds to the Dirichlet boundary
condition on $A$. \ Its iterates are denoted by $P_{k}^{A}$.
\end{definition}

\begin{definition}
The Laplace operator with Dirichlet boundary conditions on a finite set $%
A\subset \Gamma $ defined as%
\begin{equation*}
\Delta ^{A}f\left( x\right) =\left\{
\begin{array}{ccc}
\Delta f\left( x\right) & \text{if} & x\in A \\
0 & if & x\notin A%
\end{array}%
\right. .
\end{equation*}%
The smallest eigenvalue of $-\Delta ^{A}$ is denoted in general by $\lambda
(A)$ and for $A=B(x,R)$ it is denoted by $\lambda =\lambda (x,R)=\lambda
(B(x,R)).$
\end{definition}

\begin{definition}
The energy or Dirichlet form $\mathcal{E}\left( f,f\right) $ associated to
the electric network can be defined as
\begin{equation*}
\mathcal{E}\left( f,f\right) =-\left( \Delta f,f\right) =\frac{1}{2}%
\sum_{x,y\in \Gamma }\mu _{x,y}\left( f\left( x\right) -f\left( y\right)
\right) ^{2}.
\end{equation*}
\end{definition}

Using this notation the smallest eigenvalue of $-\Delta ^{A}$ can be defined
by
\begin{equation}
\lambda \left( A\right) =\inf \left\{ \frac{\mathcal{E}\left( f,f\right) }{%
\left( f,f\right) }:f\in c_{0}\left( A\right) ,f\neq 0\right\}  \label{ldef}
\end{equation}%
as well.

\begin{definition}
We introduce
\begin{equation*}
G^{A}(y,z)=\sum_{k=0}^{\infty }P_{k}^{A}(y,z)
\end{equation*}%
the local Green function, the Green function of the killed walk and the
corresponding Green's kernel as
\begin{equation*}
g^{A}(y,z)=\frac{1}{\mu \left( z\right) }G^{A}(y,z).
\end{equation*}
\end{definition}

\begin{definition}
\label{dpmv}We say that the parabolic mean value inequality$\ $holds on $%
(\Gamma ,\mu )$ if for fixed constants $0\leq c_{1}<c_{2}$ there is a $C>1$
such that for arbitrary $x\in \Gamma $, $n\in \mathbb{N}$ and $R>0,$ using
the notations $\ E=E\left( x,R\right) ,B=B\left( x,R\right) ,n=c_{2}E,\Psi =%
\left[ 0,n\right] \times B$ for any non-negative Dirichlet solution of the
heat equation
\begin{equation*}
P^{B}u_{i}=u_{i+1}
\end{equation*}%
on $\Psi ,$ the inequality
\begin{equation}
u_{n}(x)\leq \frac{C}{V(x,R)E\left( x,R\right) }\sum_{i=c_{1}E}^{n}\sum_{y%
\in B(x,R)}u_{i}(y)\mu (y)\,  \label{PMV}
\end{equation}%
holds.
\end{definition}

\begin{remark}
\label{rpmvs}Let us observe the difference between the definitions of the
parabolic and skewed parabolic mean value inequality in Definition \ref%
{dspmv} and \ref{dpmv}. As it was noted in Remark \ref{rsPMV} the skewed
parabolic mean value inequality implies $\left( VD\right) $ and $\left(
TC\right) ,$ which yields in fact the equivalence%
\begin{equation*}
\left( \ref{sPMV}\right) \Longleftrightarrow \left( \ref{PMV}\right) +\left(
VD\right) +\left( TC\right) .
\end{equation*}
\end{remark}

\begin{remark}
The above definitions of the parabolic mean value inequality and mean value
inequality can be extended to sub-solutions and the corresponding results
remain valid.
\end{remark}

\begin{definition}
We say that the a mean value property holds for the Green kernels on $%
(\Gamma ,\mu )$ \ if there is a $C>1$ such that for all $R>0,x\in \Gamma $, $%
B=B\left( x,R\right) $ and $y\in \Gamma ,d=d\left( x,y\right) >0$%
\begin{equation}
g^{B}(y,x)\leq \frac{C}{V(x,d)}\sum_{z\in B(x,d)}g^{B}(y,z)\mu (y)\,.
\label{MVG}
\end{equation}
\end{definition}

\begin{definition}
We say that the Green kernel satisfy upper bound with respect to a function $%
F$ (c.f. \cite{LW}) \ on $(\Gamma ,\mu )$ if for a $C^{\prime }>1$ there is
a $C>1$ such that for all\ $R>0$ and $y\in \Gamma ,d=:d\left( x,y\right) >0,$
$B=B\left( x,R\right) $
\begin{equation}
g^{B}(y,x)\leq \sum_{i=F(d)}^{C^{\prime }F\left( R\right) }\frac{C}{V\left(
x,f\left( i\right) \right) }  \label{UBG}
\end{equation}%
where $f\left( .\right) $ is the inverse function of $F\left( .\right) .$
\end{definition}

\begin{definition}
The Green kernel is bounded by the ratio of the mean exit time and volume on
$(\Gamma ,\mu )$ if there is a $C>1$ such that for all $R>0$ and $y\in
\Gamma ,d=:d\left( x,y\right) >0,B=B\left( x,R\right) $
\begin{equation}
g^{B}(y,x)\leq C\frac{E\left( x,R\right) }{V\left( x,d\right) }.  \label{g01}
\end{equation}
\end{definition}

\section{The local theory}

\setcounter{equation}{0}\label{sloc} In this section we shall prove the
following theorem, which implies Theorem \ \ref{tLDUE} \ and by Remark \ref%
{rpmvs}. Theorem \ref{tsPMV} as well.

\begin{theorem}
\label{tLDUE+}For a weighted graph $(\Gamma ,\mu )$ if $\left( p_{0}\right)
,(VD),\left( TC\right) $ conditions hold, then the following statements are
equivalent:\

\begin{enumerate}
\item The local diagonal upper estimate $\left( \mathbf{DUE}\right) $ holds;
there is a $C>0$ such that for all $x\in \Gamma ,$ $n>0$
\begin{equation*}
p_{n}(x,x)\leq \frac{C}{V(x,e(x,n))},
\end{equation*}

\item the upper estimate $\left( \mathbf{UE}\right) $ holds: there are $%
C,\beta >1,c>0$ such that for all $x,y\in \Gamma ,$ $n>0$%
\begin{equation*}
p_{n}\left( x,y\right) \leq \frac{C}{V\left( x,e\left( x,n\right) \right) }%
\exp \left[ -c\left( \frac{E\left( x,d\left( x,y\right) \right) }{n}\right)
^{\frac{1}{\beta -1}}\right] .
\end{equation*}

\item the parabolic mean value inequality, $\left( \ref{PMV}\right) $ holds,

\item the mean value inequality, $\left( MV\right) $ holds,

\item $\left( \ref{MVG}\right) $ holds,

\item $\left( \ref{g01}\right) $ holds.
\end{enumerate}
\end{theorem}

\begin{proposition}
\label{tLDLExy}For any weighted graph $(\Gamma ,\mu )$ if the inequality
\begin{equation}
\overline{E}(x,R)\leq CE(x,R).  \label{ebar1}
\end{equation}
holds then the local diagonal lower estimate%
\begin{equation}
P_{2n}(x,x)\geq \frac{c\mu (x)}{V(x,e(x,2n))}  \label{LDLE}
\end{equation}%
is true and
\begin{equation}
\mathbb{P}(T_{x,R}<n)\leq C\exp -c\underline{k}(x,n,R),  \label{LT}
\end{equation}%
where $\underline{k}$ is local sub-Gaussian kernel defined in $(\ref{iter}).$
\end{proposition}

The statement $\left( \ref{LT}\right) $ is given in \cite[Theorem 5.1]{TD}, $%
\left( \ref{LDLE}\right) $ in \cite[Proposition 6.4, 6.5]{TD}.

\begin{remark}
It is worth to observe that the diagonal lower estimate $\left( \ref{LDLE}%
\right) $ and the diagonal upper estimate $\left( \ref{LDUE}\right) $
matches up to a constant.
\end{remark}

\subsection{Properties of the mean exit time}

In this section we recall some results from \cite{ter} which describe the
behavior of the mean exit time in the local framework. The first one is the
Einstein relation below in its multiplicative form which plays an important
role.

\begin{theorem}
\label{tLER}\label{tER1}If $\left( p_{0}\right) ,(VD)$ \ and $\left(
TC\right) $ hold then $\left( ER\right) ,$ the Einstein relation
\begin{equation}
E(x,2R)\simeq \rho (x,R,2R)v(x,R,2R)  \label{ER}
\end{equation}%
holds.
\end{theorem}

For the proof see \cite{ter}.

\begin{lemma}
\label{lmarkov}On all $\left( \Gamma ,\mu \right) $ for $\ $any $x\in \Gamma
,R>S>0$%
\begin{equation*}
E\left( x,R+S\right) \geq E\left( x,R\right) +\min_{y\in S\left( x,R\right)
}E\left( y,S\right) .
\end{equation*}
\end{lemma}

\begin{proof}
Let us denote $A=B\left( x,R\right) ,B=B\left( x,R+S\right) $. First let us
observe that from the triangular inequality it follows that for any $y\in
S\left( x,R\right) $
\begin{equation*}
B\left( y,S\right) \subset B\left( x,R+S\right) .
\end{equation*}%
From this and from the strong Markov property one obtains that%
\begin{eqnarray*}
E\left( x,R+S\right) &=&E_{x}\left( T_{B}+E_{X_{T_{B}}}\left( x,R+S\right)
\right) \\
&\geq &E\left( x,R\right) +E_{x}\left( E_{X_{T_{B}}}\left(
X_{T_{B}},S\right) \right) .
\end{eqnarray*}%
But $X_{T_{B}}\in S\left( x,R\right) $ which gives the statement.
\end{proof}

\begin{corollary}
\label{cmonot}The mean exit time is strictly increasing in $R\in
\mathbb{N}
$ and hence $E\left( x,R\right) $ has an inverse in the second variable
\begin{equation*}
e\left( x,n\right) =\min \left\{ R\in
\mathbb{N}
:E\left( x,R\right) \geq n\right\} .
\end{equation*}
\end{corollary}

\begin{proof}
From Lemma \ref{lmarkov} and $E\left( x,1\right) \geq 1$ it follows that
\begin{equation}
E\left( x,R+1\right) \geq E\left( x,R\right) +1.  \label{e>e+1}
\end{equation}
\end{proof}

\begin{definition}
We shall say that the mean exit time has the anti doubling property if there
is an $A_{E}>1$ such that for all $x\in \Gamma ,R>0$
\begin{equation}
E(x,A_{E}R)\geq 2E(x,R).  \label{aTD1}
\end{equation}
\end{definition}

\begin{proposition}
\label{pPD3E}If $\left( p_{0}\right) $\ and $\left( TC\right) $ hold then $%
\left( \ref{aTD1}\right) ,$ anti-doubling for the mean exit time holds.
\end{proposition}

\begin{proof}
In \cite{ter} it is shown that $\left( \ref{pd2e}\right) $ implies $\left( %
\ref{aTD1}\right) $ but from $\left( TC\right) $ the inequality $\left( \ref%
{pd2e}\right) $ follows.
\end{proof}

\bigskip The anti-doubling property of the mean exit time is equivalent with
the existence of $c,\beta ^{\prime }>0$ such that%
\begin{equation}
\frac{E\left( x,R\right) }{E\left( x,S\right) }\geq c\left( \frac{R}{S}%
\right) ^{\beta ^{\prime }}  \label{adE}
\end{equation}%
for all, $R>S>0,x\in \Gamma ,y\in B\left( x,R\right) .$

Sometimes we will refer to the pair of doubling and anti-doubling property
as doubling properties. The properties of the inverse function $e$ of $E$
(which exists by Corollary \ref{cmonot} ) and properties of $E$ are linked
as the following evident lemma states.

\begin{lemma}
The following statements are equivalent \newline
\newline
1. There are $C,c>0,\beta \geq \beta ^{\prime }>0$ such that for all $x\in
\Gamma ,R\geq S>0,$ $y\in B\left( x,R\right) $%
\begin{equation}
c\left( \frac{R}{S}\right) ^{\beta ^{\prime }}\leq \frac{E\left( x,R\right)
}{E\left( y,S\right) }\leq C\left( \frac{R}{S}\right) ^{\beta },  \label{Eb}
\end{equation}%
2. There are $C,c>0,\beta \geq \beta ^{\prime }>0$ such that for all $x\in
\Gamma ,n\geq m>0,$ $y\in B\left( x,e\left( x,n\right) \right) $%
\begin{equation}
c\left( \frac{n}{m}\right) ^{1/\beta }\leq \frac{e\left( x,n\right) }{%
e\left( y,m\right) }\leq C\left( \frac{n}{m}\right) ^{1/\beta ^{\prime }}.
\label{eb}
\end{equation}
\end{lemma}

\begin{remark}
Let us recall that the doubling property and the anti-doubling property of $%
E $ is equivalent with the right and left hand side of \ref{Eb} for $y=x.$
\end{remark}

The following corollary is from \cite{ter}.

\begin{corollary}
\label{ce>2 copy(1)}\label{cb>2}If $\left( p_{0}\right) $,$(VD)$ and
\begin{equation}
\max_{y\in B\left( x,R\right) }E_{y}\left( x,R\right) \leq CE\left(
x,R\right)  \label{Ebar}
\end{equation}%
hold then%
\begin{equation}
E(x,R)\geq cR^{2},  \label{e>r2}
\end{equation}%
i.e.,
\begin{equation}
\beta \geq 2.  \label{bb>2}
\end{equation}
\end{corollary}

\begin{remark}
\label{re>r2} In the present context we need a weaker statement, $\left(
p_{0}\right) +\left( VD\right) +\left( TC\right) \Longrightarrow \left( \ref%
{e>r2}\right) $,$\left( \ref{bb>2}\right) $. \ This immediately follows from
Theorem \ref{tER1} and $\left( \ref{rv>2}\right) $ and the fact that $\left(
TC\right) $ implies $\left( \ref{Ebar}\right) $.
\end{remark}

\begin{lemma}
\label{k> copy(1)} If $\left( p_{0}\right) $,$(VD)$ and $\left( TC\right) $
hold, then for$\ \underline{k}=\ \underline{k}(x,n,R)$ defined in $\left( %
\ref{iter}\right) $
\begin{equation}
\underline{k}+1\geq c\left( \frac{E(x,R)}{n}\right) ^{\frac{1}{\beta -1}}%
\text{\ }.  \label{k>e/n}
\end{equation}%
for all $x\in \Gamma ,R,n>0$ for fixed $c>0,\beta >1.$
\end{lemma}

\begin{proof}
The statement follows from $\left( TC\right) $ easily, $\beta >1$ is ensured
by $\beta \geq 2$ from Corollary \ref{cb>2}.
\end{proof}

\subsection{The resolvent}

In this section we recall from \cite{GT2} the key intermediate step to prove
the diagonal upper estimate. We introduce for a finite set $A\subset \Gamma $
the $\lambda ,m-$resolvent
\begin{equation*}
G_{\lambda ,m}^{A}=\left( \left( \lambda +1\right) I-P^{A}\right) ^{-m}
\end{equation*}%
for $\lambda \geq 0,m\geq 0$ and let us define \ the kernel corresponding to
the resolvent as
\begin{equation*}
g_{\lambda ,m}^{A}\left( x,y\right) =\frac{1}{\mu \left( y\right) }%
G_{\lambda ,m}^{A}\left( x,y\right) .
\end{equation*}

\begin{theorem}
\label{tgl}Assuming $\left( p_{0}\right) ,(VD)$ and $\left( TC\right) $ the
condition $\left( \ref{g01}\right) $ implies, for a large enough $m>1$ and
for all $0<\lambda <1$, $x\in \Gamma ,$ the inequality $:$
\begin{equation}
g_{\lambda ,m}(x,x)\leq C\frac{\lambda ^{-m}}{V(x,e(x,\lambda ^{-1}))}.
\label{gllm}
\end{equation}
\end{theorem}

The proof closely follows the corresponding proof of \cite[Theorem 5.7]{GT2}
so we omit it. One should reproduce it simply replacing the space-time scale
function $R^{\beta }$ by $E\left( x,R\right) $ and using the doubling
properties repeatedly.

\subsection{The local diagonal upper estimate}

We start with the application of the $\lambda,m-$resolvent bound to obtain
the local diagonal upper estimate.

\begin{theorem}
\label{tG-->LDUE}The conditions $\left( p_{0}\right) ,\left( VD\right)
,\left( TC\right) $ and $\left( \ref{gllm}\right) $ imply $\left( DUE\right)
,$
\begin{equation*}
p_{n}(x,x)\leq \frac{C}{V(x,e(x,n))}.
\end{equation*}
\end{theorem}

Again the proof is easy modification of \cite[Theorem 6.1]{GT2} therefore we
skip it.

\begin{lemma}
\label{lgs}If $\left( p_{0}\right) ,(VD)$ and $\left( TC\right) $ hold then
\begin{equation*}
\left( \ref{MVG}\right) \Leftrightarrow \left( MV\right) \Longrightarrow
\left( \ref{g01}\right) .
\end{equation*}
\end{lemma}

\begin{proof}
First we show $\left( \ref{MVG}\right) \Longrightarrow \left( MV\right) .$
Denote $B=B\left( x,R\right) ,$ $U=B(x,2R).$

Let $u\geq 0$ a harmonic function on $B(x,R)$ \ and consider it's
representation:
\begin{equation*}
u\left( z\right) =\sum_{w\in U}g^{U}\left( z,w\right) \nu \left( w\right)
\end{equation*}%
which always exists with a $\nu \geq 0,\nu \in c_{0}\left( U\right) $ charge
(the standard construction can be reproduced following the proof of \cite[%
Lemma 10.2]{GT1}). \ Applying this decomposition and $\left( \ref{MVG}%
\right) $ to $u\left( x\right) $ the mean value inequality follows.
\begin{align*}
u\left( x\right) & =\sum_{w\in U}g^{U}\left( x,w\right) \nu \left( w\right)
\leq \frac{C}{V\left( x,R\right) }\sum_{w\in U}\sum_{z\in B}g^{U}\left(
z,w\right) \nu \left( w\right) \mu \left( z\right) \\
& =\frac{C}{V\left( x,R/2\right) }\sum_{z\in B}\sum_{w\in U}g^{U}\left(
z,w\right) \nu \left( w\right) \mu \left( z\right) \leq \frac{C}{V\left(
x,R\right) }\sum_{z\in B}u\left( z\right) \mu \left( z\right) .
\end{align*}%
\ \ The opposite implication $(MV)\Longrightarrow \left( \ref{MVG}\right) $
follows simply applying $(MV)$ to $u\left( x\right) =g^{U}\left( x,w\right) $%
. Finally $\left( MV\right) \Longrightarrow $ $\left( \ref{g01}\right) $ is
immediate. If $d=d\left( x,y\right) >R$ then $g^{B(x,R)}(x,y)=0$ and there
is nothing to prove. Otherwise, consider the function $u(z)=g^{B(x,2R)}(y,z)$%
. This function is non-negative and harmonic in the ball $B(x,d)$. Hence, by
$(MV),\left( VD\right) $ and $\left( TC\right) $%
\begin{equation*}
u(x)\leq \frac{C}{V(x,d)}\sum_{z\in B(x,d)}u(z)\mu (z)\leq \frac{C}{V(x,d)}%
\overline{E}(x,2R)\leq C\frac{E(x,R)}{V(x,d)}.
\end{equation*}%
Finally $(\ref{g01})$ follows from $g^{B(x,R)}\leq g^{B(x,2R)}$.
\end{proof}

\subsection{From $\left( DUE\right) $ to $\left( UE\right) $}

The proof is easy modification of the nice argument given in \cite{Gfk} for
the corresponding implication. \

\begin{lemma}
Let $r=\frac{1}{2}d\left( x,y\right) $ then
\begin{equation}
p_{2n}\left( x,y\right) \leq P_{x}\left( T_{x,r}<n\right) \max_{\substack{ %
n\leq k\leq 2n  \\ v\in \partial B\left( x,r\right) }}p_{k}\left( v,y\right)
+P_{y}\left( T_{y,r}<n\right) \max_{\substack{ n\leq k\leq 2n  \\ z\in
\partial B\left( y,r\right) }}p_{k}\left( z,x\right) .  \label{pcut}
\end{equation}
\end{lemma}

\begin{proof}
The statement follows from the first exit decompositions staring from $x$
(and from $y$ respectively) and from the Markov property as in \cite{Gfk}.
\end{proof}

\begin{theorem}
\label{tldue-ue}$\left( p_{0}\right) +\left( VD\right) +\left( TC\right)
+\left( DUE\right) \Longrightarrow \left( UE\right) .$
\end{theorem}

\begin{lemma}
\label{lvv}If $\left( p_{0}\right) ,\left( VD\right) $ and $\left( TC\right)
$ hold then for all $\varepsilon >0$ there are $C_{\varepsilon },C>0$ such
that for all $k>0,y,z\in \Gamma ,r=d\left( y,z\right) $%
\begin{equation*}
\sqrt{\frac{V\left( y,e\left( y,k\right) \right) }{V\left( v,e\left(
z,k\right) \right) }}\leq C_{\varepsilon }\exp \varepsilon C\left( \frac{%
E\left( y,r\right) }{k}\right) ^{\frac{1}{\left( \beta -1\right) }}.
\end{equation*}
\end{lemma}

\begin{proof}
Let us consider the minimal $m$ for which $e\left( y,m\right) \geq r,$%
\begin{equation*}
e\left( y,k\right) \leq e\left( y,k+m\right)
\end{equation*}%
and use the anti doubling property with $\beta ^{\prime }>0$ to obtain%
\begin{eqnarray*}
\sqrt{\frac{V\left( y,e\left( y,k\right) \right) }{V\left( z,e\left(
z,k\right) \right) }} &\leq &\sqrt{\frac{V\left( y,e\left( y,k+m\right)
\right) }{V\left( z,e\left( z,k\right) \right) }} \\
&\leq &C\left( \frac{e\left( y,k+m\right) }{e\left( z,k\right) }\right)
^{\alpha /2}\leq C\left( \frac{k+m}{k}\right) ^{\frac{\alpha }{2\beta
^{\prime }}}=C\left( 1+\frac{m-1+1}{k}\right) ^{\frac{\alpha }{2\beta
^{\prime }}} \\
&\leq &C\left( 1+\frac{E\left( y,r\right) +1}{k}\right) ^{\frac{\alpha }{%
2\beta ^{\prime }}} \\
&\leq &C_{\varepsilon }\exp \varepsilon C\left( \frac{E\left( y,r\right) }{k}%
\right) ^{\frac{1}{\left( \beta -1\right) }}.
\end{eqnarray*}%
Here we have to note that by Remark \ref{re>r2} it follows that $\beta >1$
furthermore from the conditions that $\alpha ,\beta ^{\prime }>0.$ The
manipulation of the exponents used the trivial estimate $1+x^{\frac{a}{a}%
}\leq \left( 1+x^{\frac{1}{a}}\right) ^{a},$ where $x,a>0$. As a result we
obtain by repeated application of $\left( TC\right) $ that
\begin{equation*}
\sqrt{\frac{V\left( y,e\left( y,k\right) +r\right) }{V\left( z,e\left(
z,k\right) \right) }}\leq C_{\varepsilon }\exp \varepsilon C\left( \frac{%
E\left( y,r\right) }{k}\right) ^{\frac{1}{\left( \beta -1\right) }}.
\end{equation*}
\end{proof}

\begin{proof}[Proof of Theorem \protect\ref{tldue-ue}]
If $d\left( x,y\right) \leq 2$ then the statement follows from $\left(
p_{0}\right) $ according to Remark \ref{re<1/p}. We use $\left( \ref{pcut}%
\right) $ with $r=\frac{1}{2}d\left( x,y\right) $ and start to handle the
first term in%
\begin{equation*}
P_{x}\left( T_{x,r}<n\right) \max_{\substack{ n\leq k\leq 2n  \\ v\in
\partial B\left( x,r\right) }}p_{k}\left( v,y\right) .
\end{equation*}%
Let us recall that from $\left( TC\right) $ it follows that
\begin{equation}
\mathbb{P}(T_{x,r}<n)\leq C\exp -c\underline{k}(x,n,r),
\end{equation}%
and use $r\leq d\left( v,y\right) \leq 3r$ furthermore $\left( \ref{k>e/n}%
\right) $ to get%
\begin{equation*}
P_{x}\left( T_{x,r}<n\right) \leq C\exp \left[ -c\left( \frac{E\left(
x,r\right) }{n}\right) ^{\frac{1}{\beta -1}}\right] .
\end{equation*}%
Let us treat the other term. First we observe that
\begin{eqnarray}
p_{2k+1}\left( y,v\right) &\leq &\sum_{z\sim y}P_{2k}\left( y,z\right)
P\left( z,v\right) \frac{1}{\mu \left( v\right) }  \label{2k+1} \\
&=&\sum_{z\sim y}P_{2k}\left( y,z\right) P\left( v,z\right) \frac{1}{\mu
\left( z\right) }  \notag \\
&\leq &\max_{z\sim y}p_{2k}\left( y,z\right) \sum_{z\sim v}P\left( v,z\right)
\notag \\
&=&\max_{z\sim y}p_{2k}\left( y,z\right) .  \notag
\end{eqnarray}%
and recall that
\begin{equation*}
p_{2k}\left( x,y\right) \leq \sqrt{p_{2k}\left( x,x\right) p_{2k}\left(
y,y\right) },
\end{equation*}%
which yields using the doubling properties of $V$, $E$ and for $w\sim v$ $%
d\left( y,v\right) \simeq d\left( y,w\right) $ (provided $v,w\neq y)$ that
\begin{eqnarray}
&&\max_{\substack{ n\leq k\leq 2n  \\ w\in \partial B\left( x,r\right) }}%
p_{k}\left( w,y\right)  \label{maxp} \\
&\leq &\max_{\substack{ n\leq 2k\leq 2n  \\ v\sim w\in \partial B\left(
x,r\right) }}p_{2k}\left( v,y\right) \\
&\leq &\max_{_{\substack{ n\leq 2k\leq 2n  \\ v\sim w\in \partial B\left(
x,r\right) }}}\frac{C}{\sqrt{V\left( y,e\left( y,2k\right) \right) V\left(
v,e\left( v,2k\right) \right) }} \\
&\leq &\max_{_{\substack{ n\leq 2k\leq 2n  \\ v\sim w\in \partial B\left(
x,r\right) }}}\frac{C}{V\left( y,e\left( y,n\right) \right) }\sqrt{\frac{%
V\left( y,e\left( y,n\right) \right) }{V\left( v,e\left( v,n\right) \right) }%
},  \notag
\end{eqnarray}%
Let us observe that $d\left( v,y\right) \leq 3r+2\leq 5r$ if $r\geq 1$ and
apply Lemma \ref{lvv} to proceed with
\begin{eqnarray*}
&&\max_{_{\substack{ n\leq 2k\leq 2n  \\ v\sim w\in \partial B\left(
x,r\right) }}}p_{2k}\left( v,y\right) P_{x}\left( T_{x,r}<n\right) \\
&\leq &\frac{C}{V\left( y,e\left( y,n\right) \right) }C_{\varepsilon }\exp %
\left[ \varepsilon C\left( \frac{E\left( y,5r\right) }{n}\right) ^{\frac{1}{%
\left( \beta -1\right) }}-c\left( \frac{E\left( x,r\right) }{n}\right) ^{%
\frac{1}{\left( \beta -1\right) }}\right] ,
\end{eqnarray*}%
choosing $\varepsilon $ small enough and applying $\left( TC\right) $ we
have the inequality
\begin{eqnarray}
&&\max_{\substack{ n\leq k\leq 2n  \\ v\in \partial B\left( x,r\right) }}%
p_{k}\left( v,y\right) \exp \left[ -c\left( \frac{E\left( x,d\left(
x,y\right) \right) }{n}\right) ^{\frac{1}{\left( \beta -1\right) }}\right]
\notag \\
&\leq &\frac{C}{V\left( y,e\left( y,n\right) \right) }\exp \left[ -c\left(
\frac{E\left( x,d\left( x,y\right) \right) }{n}\right) ^{\frac{1}{\left(
\beta -1\right) }}\right] .  \notag
\end{eqnarray}%
By symmetry one gets%
\begin{eqnarray*}
p_{2n}\left( x,y\right) &\leq &C\left( \frac{1}{V\left( x,e\left( x,n\right)
\right) }+\frac{1}{V\left( y,e\left( y,n\right) \right) }\right) \exp \left[
-c\left( \frac{E\left( x,d\left( x,y\right) \right) }{n}\right) ^{\frac{1}{%
\left( \beta -1\right) }}\right] \\
&=&\frac{C}{V\left( x,e\left( x,n\right) \right) }\left( 1+\frac{V\left(
x,e\left( x,n\right) \right) }{V\left( y,e\left( y,n\right) \right) }\right)
\exp \left[ -c\left( \frac{E\left( x,d\left( x,y\right) \right) }{n}\right)
^{\frac{1}{\left( \beta -1\right) }}\right] .
\end{eqnarray*}%
Now we use Lemma \ref{lvv} again obtain%
\begin{equation*}
\frac{V\left( x,e\left( x,n\right) \right) }{V\left( y,e\left( y,n\right)
\right) }\leq C_{\varepsilon }\exp \varepsilon C\left( \frac{E\left(
x,2r\right) }{n}\right) ^{\frac{1}{\left( \beta -1\right) }}
\end{equation*}%
and $\varepsilon $ can be chosen to satisfy $\varepsilon C<\frac{c}{2}$ to
receive%
\begin{equation*}
\left( 1+\exp \left[ \left( \varepsilon C-\frac{c}{2}\right) \left( \frac{%
E\left( x,d\left( x,y\right) \right) }{n}\right) ^{\frac{1}{\left( \beta
-1\right) }}\right] \right) \leq 2
\end{equation*}%
\begin{equation*}
p_{2n}\left( x,y\right) \leq \frac{2C}{V\left( x,e\left( x,n\right) \right) }%
\exp \left[ -\frac{c}{2}\left( \frac{E\left( x,d\left( x,y\right) \right) }{n%
}\right) ^{\frac{1}{\left( \beta -1\right) }}\right] ,
\end{equation*}%
which is the needed estimate for even $n.$ For odd number of steps the
results follows using for $x\neq y$ the trivial inequality $\left( \ref{2k+1}%
\right) $ and $d\left( x,y\right) \simeq d\left( x,z\right) $ if $z\neq
x,y\sim z.$ In particular if the maximum in $\left( \ref{2k+1}\right) $
attained at $x=z$ then the statement follows from $\left( DUE\right) $ and $%
\left( p_{0}\right) .$
\end{proof}

\begin{remark}
With a slight modification of the beginning of the proof one can get
\begin{equation}
p_{n}\left( x,y\right) \leq \frac{C\exp \left( -c\underline{k}\left(
y,n,r\right) \right) }{V\left( x,e\left( x,n\right) \right) }+\frac{C\exp
\left( -c\underline{k}\left( x,n,r\right) \right) }{V\left( y,e\left(
y,n\right) \right) }  \label{ue(k)}
\end{equation}%
where $r=\frac{1}{2}d\left( x,y\right) ,$ which is sharper then the above
upper estimate. Let us note that our deduction shows that $\left( \ref{ue(k)}%
\right) $ is equivalent with the upper estimate.
\end{remark}

\begin{proof}[Proof or Theorem \protect\ref{tLDUE+}]
Let us assume that the conditions $\left( p_{0}\right) ,$ $(VD),$ $\left(
TC\right) $ hold. From Theorem \ref{tgl} and \ref{tG-->LDUE} it follows that
$\left( \ref{g01}\right) \Longrightarrow \left( \ref{gllm}\right)
\Longrightarrow \left( DUE\right) ,$ which covers the implication $%
6.\Longrightarrow 1.$ From Lemma \ref{lgs} we know that $(MV)$ $%
\Longleftrightarrow $ $(\ref{MVG})$ and $\left( MV\right) \Longrightarrow
\left( \ref{g01}\right) $ i.e. $4.\Longleftrightarrow 5.\Longrightarrow 6.$
In Theorem \ref{tldue-ue} we have shown that $\left( DUE\right)
\Longrightarrow \left( UE\right) ,$ which means $1\Longrightarrow 2$ while
the reverse implication is trivial. The parabolic mean value inequality, $%
\left( \ref{PMV}\right) $ implies $(MV),$ \ i.e. $3\Longrightarrow 4$. \ It
is left to show $\ $that $2\Longrightarrow 3.$ i.e. $\left( UE\right)
\Longrightarrow \left( \ref{PMV}\right) $. We shall show a little bit more.\
\ Let us consider a Dirichlet solution $u_{i}\left( w\right) \geq $ $0$ on $%
B\left( x,R\right) $ with initial data $u_{0}\in c_{0}\left( B\left(
x,R\right) \right) $. Denote $E=E\left( x,R\right) .$ \ Consider $%
v_{i}\left( z\right) =\mu \left( z\right) u_{i}\left( z\right) $, $%
0<c_{1}<c_{2}<c_{3}<c_{4}$ and $n\in \left[ c_{3}E,c_{4}E\right] ,j\in \left[
c_{1}E,c_{2}E\right] $. By definition
\begin{equation*}
u_{n}\left( w\right) =\sum_{y\in \Gamma }P_{n-j}^{B}\left( w,y\right)
u_{j}\left( y\right)
\end{equation*}%
and
\begin{equation*}
v_{n}\left( w\right) =\sum_{y\in B\left( x,R\right) }P_{n-j}^{B}\left(
y,w\right) v_{j}\left( y\right) \leq \mu \left( w\right) \max\limits_{y\in
B}p_{n-j}^{B}\left( y,w\right) \sum_{y\in B\left( x,R\right) }v_{j}\left(
y\right) ,
\end{equation*}%
from which one has by $p^{B}\leq p$ and $\left( UE\right) $ that
\begin{equation*}
u_{n}\left( w\right) \leq \max\limits_{y\in B}p_{n-j}^{B}\left( w,y\right)
\sum_{y\in B\left( x,R\right) }u_{j}\left( y\right) \mu \left( y\right) \leq
\frac{C}{V\left( w,e\left( w,n-j\right) \right) }\sum_{y\in B\left(
x,R\right) }u_{j}\left( y\right) \mu \left( y\right) .
\end{equation*}%
Using the doubling properties of $e$ and $V$ it follows that $V\left(
w,e\left( w,n-j\right) \right) \simeq V\left( x,R\right) $ and
\begin{equation}
u_{n}\left( w\right) \leq \frac{C}{V\left( x,R\right) }\sum_{y\in B\left(
x,R\right) }u_{j}\left( y\right) \mu \left( y\right) .  \label{presum}
\end{equation}%
Finally summing $\left( \ref{presum}\right) $ for $j\in \left[ c_{1}E,c_{2}E%
\right] $ we obtain
\begin{equation*}
u_{n}\left( w\right) \leq \frac{C}{E\left( x,R\right) V\left( x,R\right) }%
\sum_{j=c_{1}E}^{c_{2}E}\sum_{y\in B\left( x,R\right) }u_{j}\left( y\right)
\mu \left( y\right) .
\end{equation*}%
This means that this inequality holds for all $\left( n,w\right) \in \lbrack
c_{3}E,c_{4}E]\times B\left( x,R\right) =\Psi $, e.g. using the properties
of $V$ and $E$ again, for all $y\in V\left( x,R\right) $%
\begin{equation}
\max\limits_{\Psi }u\leq \frac{C}{E\left( y,2R\right) V\left( y,2R\right) }%
\sum_{j=c_{1}E}^{c_{2}E}\sum_{y\in B\left( x,R\right) }u_{j}\left( y\right)
\mu \left( y\right) .  \label{smaxPMV}
\end{equation}%
also satisfied. \ It is clear that $\left( \ref{smaxPMV}\right) $ implies $%
\left( \ref{sPMV}\right) $ and $\left( \ref{PMV}\right) $ as well which
finishes the proof.
\end{proof}

\section{Semi-local theory}

\setcounter{equation}{0}\label{ssloc} This section is split into two parts.
\ In the first part the reformulation and extension of the upper estimates
are developed. \ In this part typically we work under the assumptions of $%
\left( VD\right) ,$ $\left( TD\right) $ and $\left( E\right) .$ \ In the
second part, in Section \ref{stwoside}, we discuss the two-sided estimate. \
There the main assumptions are $\left( VD\right) $,$\left( E\right) $ and
the elliptic Harnack inequality (to be defined there).

\subsection{The upper estimate}

Let us start with the definition of the $F$-parabolic mean value inequality.

\begin{definition}
We shall say that the $F-$parabolic mean value inequality holds if for the
function $F\left( R\right) =\inf_{x\in \Gamma }E\left( x,R\right)
,c_{_{2}}>c_{1}>0$ constants there is a $C>1$ such that for all $R>0,x\in
\Gamma $ for all non-negative Dirichlet solutions $u_{n}$ of the discrete
heat equation
\begin{equation*}
P^{B\left( x,R\right) }u_{n}=u_{n+1}
\end{equation*}%
on $\left[ 0,c_{2}E\left( x,R\right) \right] \times B\left( x,R\right) $%
\begin{equation}
u_{n}(x)\leq \frac{C}{V(x,2R)E\left( x,2R\right) }\sum_{i=c_{1}F}^{c_{2}F}%
\sum_{z\in B(x,R)}u_{i}(z)\mu (z)\,  \label{sPMVF}
\end{equation}%
satisfied, where $F=F\left( R\right) $,$n=c_{2}F\left( R\right) $.
\end{definition}

\begin{remark}
\label{rSPMVF}Let us observe that in this definition the volume doubling
property and time comparison principle are "built in", as in the skewed
parabolic mean value inequality. \ The condition $E\simeq F$ follows from %
\ref{sPMVF} as well.
\end{remark}

In this section we prove the following theorems.

\begin{theorem}
\label{tslUE1}\label{tDUE+}For any weighted graph $(\Gamma ,\mu )$ if $%
\left( p_{0}\right) ,(VD),\left( TD\right) $ and $\left( E\right) $ hold
then the following statements are equivalent

\begin{enumerate}
\item for a proper function $F$, the $F$-based diagonal upper estimate hold,
that is, there is a $C>0$ such that for all $x\in \Gamma ,n>0$
\begin{equation}
P_{n}(x,x)\leq \frac{C\mu (x)}{V(x,f(n))},  \label{duef}
\end{equation}

\item the estimate, $\left( UE_{F}\right) $ holds for a proper $F$: there
are $C,c>0$ such that for all $x,y\in \Gamma ,n>0$%
\begin{equation}
P_{n}(x,y)\leq \frac{C\mu (y)}{V(x,f(n))}\exp -ck(n,d(x,y)),  \label{UEF1}
\end{equation}

\item the parabolic mean value inequality, $\left( \ref{PMV}\right) $ holds,

\item the mean value inequality, $\left( MV\right) $ holds,

\item $\left( \ref{MVG}\right) $ holds,

\item $\left( \ref{UBG}\right) $ holds,

\item $\left( \ref{g01}\right) $ holds.
\end{enumerate}
\end{theorem}

For the notion of (very-) proper $F$ see Definition \ref{dpropf}, and the
existence of the inverse of $F$ in the next section. Similarly to Theorem %
\ref{tsPMV} the following is true.

\begin{theorem}
\label{tsPMVF} If $\left( \Gamma ,\mu \right) $ satisfies $\left(
p_{0}\right) $ then the following conditions are equivalent.

\begin{enumerate}
\item the $F$parabolic mean value inequality, $\left( \ref{sPMVF}%
\right) $ holds for a proper $F,$

\item $\left( VD\right) ,\left( TD\right) ,\left( E\right) $ and $\left(
MV\right) $ holds

\item $\left( VD\right) ,\left( TD\right) ,\left( E\right) $ and $\left( \ref%
{duef}\right) $ holds,

\item $\left( VD\right) ,\left( TD\right) ,\left( E\right) $ and $\left(
UE_{F}\right) $ holds,
\end{enumerate}
\end{theorem}

\subsection{The properties of the scale function}

\label{sscalef}Let us recall that $\left( TD\right) $ $+$ $\left( E\right)
\Longrightarrow \left( TC\right) $ and consequently we can deduce several
properties of the space-time scale function easily. First of all the
Einstein relations holds under the standing assumptions of this section.

\begin{corollary}
If $\left( \Gamma ,\mu \right) $ satisfies $\left( p_{0}\right) ,\left(
VD\right) ,\left( TD\right) $ and $\left( E\right) $ then the Einstein
relation
\begin{equation*}
E(x,2R)\simeq \rho (x,R,2R)v(x,R,2R)
\end{equation*}%
holds.
\end{corollary}

The statement follows from Theorem \ref{tLER} since $\left( TD\right) $ $+$ $%
\left( E\right) \Longrightarrow \left( TC\right) .$

From the time doubling property it follows that the function
\begin{equation}
F\left( R\right) =\inf_{x\in \Gamma }E\left( x,R\right)  \label{fdef}
\end{equation}%
also has doubling property:%
\begin{equation}
F(2R)\leq D_{E}F(R),  \label{Fd}
\end{equation}%
in particular it is also clear that
\begin{equation}
\frac{F(R)}{F(S)}\leq C_{F}\left( \frac{R}{S}\right) ^{\beta },  \label{ED2F}
\end{equation}%
holds, where $\beta =\log _{2}D_{E}.$

\begin{corollary}
\label{cinv}If $\left( E\right) $ holds and $F\left( R\right) =\inf_{x\in
\Gamma }E\left( x,R\right) $, then $F\left( R\right) $ is strictly
increasing in $R\in
\mathbb{N}
$ and has an inverse.
\end{corollary}

\begin{proof}
The statement follows from\ $\left( \ref{e>e+1}\right) $, simply choose $x$
for which%
\begin{eqnarray*}
F\left( R+1\right) &\geq &E\left( x,R+1\right) -\frac{1}{2} \\
&\geq &E\left( x,R\right) +1-\frac{1}{2}>F\left( R\right) .
\end{eqnarray*}
\end{proof}

\begin{corollary}
\label{cFlin}If $\left( E\right) $ holds and $F\left( R\right) =\inf_{x\in
\Gamma }E\left( x,R\right) ,$ then for all $L,R,S\in
\mathbb{N}
,$ and $R>S>0$%
\begin{equation}
F\left( R+S\right) \geq F(R)+F(S)  \label{fsupeaddit}
\end{equation}%
and%
\begin{equation}
F(LR)\geq LF(R).  \label{ED4F}
\end{equation}
\end{corollary}

\begin{proof}
Both statements are immediate from Lemma \ref{lmarkov} using the same
argument as in Corollary \ref{cinv}.
\end{proof}

\begin{definition}
We shall say that $F$ has the anti doubling property if there is a $%
A_{F},B_{F}>1$ such that
\begin{equation}
F(A_{F}R)\geq B_{F}F(R).  \label{adF}
\end{equation}%
and the strong anti-doubling property if $B_{F}>A_{F}.$
\end{definition}

\begin{remark}
Equivalently the anti-doubling property for $F$ means that there are $%
c,\beta ^{^{\prime }}>0$ such that for $R>S>0$%
\begin{equation}
\frac{F\left( R\right) }{F\left( S\right) }\geq c\left( \frac{R}{S}\right)
^{\beta ^{\prime }}  \label{aDF1}
\end{equation}%
and the strong anti-doubling property is equivalent with $\left( \ref{aDF1}%
\right) $ for a $\beta ^{\prime }>1.$
\end{remark}

\begin{proposition}
If $\left( \Gamma ,\mu \right) $ satisfies $\left( p_{0}\right) $ and $%
\left( E\right) ,$ then for the function defined in $\left( \ref{fdef}%
\right) $ the anti-doubling property $\left( \ref{adF}\right) $ holds.
\end{proposition}

\begin{proof}
Since $\left( E\right) \Longrightarrow \left( \ref{pd2e}\right) $ by
Proposition \ref{pPD3E} we have%
\begin{equation*}
E(x,AR)\geq 2E(x,R).
\end{equation*}%
and it is clear that for any $\varepsilon >0$,$R>0$ there is an $x$ for
which \
\begin{eqnarray*}
F\left( AR\right) &\geq &E\left( x,AR\right) -\varepsilon \geq 2E\left(
x,R\right) -\varepsilon \\
&\geq &2F\left( R\right) -\varepsilon
\end{eqnarray*}%
which yields the statement since $\varepsilon $ is arbitrarily small.
\end{proof}

\begin{corollary}
\label{cF>2}If $\left( p_{0}\right) $,$(VD),\left( TD\right) $ and $(E)$
holds then
\begin{equation*}
E(x,R)\geq cR^{2}
\end{equation*}%
and%
\begin{equation}
F\left( R\right) \geq cR^{2}.  \label{F>R2}
\end{equation}
\end{corollary}

\begin{proof}
\bigskip The statement follows from Remark \ref{re>r2}.
\end{proof}

\begin{definition}
\label{dpropf}A function $F:%
\mathbb{N}
\rightarrow
\mathbb{R}
$ will be called proper if it is strictly monotone and satisfies $\left( \ref%
{Fd}\right) ,\left( \ref{fsupeaddit}\right) ,\left( \ref{adF}\right) $ and $%
\left( \ref{F>R2}\right) $, and very proper if in addition it satisfies $%
\left( \ref{adF}\right) $ with a $B_{F}>A_{F}.$
\end{definition}

The above observations can be summarized as follows.

\begin{corollary}
\label{cFproper}If $\left( \Gamma ,\mu \right) $ satisfies $\left(
p_{0}\right) ,\left( VD\right) ,\left( TD\right) $ and $\left( E\right) $
then $F$ is proper.
\end{corollary}

The following lemma provides estimates of the sub-Gaussian kernel function.

\begin{lemma}
\label{k>} If $\left( E\right) $ and $\left( TD\right) $ hold, then for $%
k=k(n,R)$%
\begin{equation}
k+1\geq c\left( \frac{F(R)}{n}\right) ^{\frac{1}{\beta -1}}\text{\ \ ,\ \ \
\ }k+1\geq c^{\prime }\left( \frac{R}{f(n)}\right) ^{\frac{\beta }{\beta -1}}%
\text{\ \ and \ \ }k\leq C\left( \frac{R^{\beta }}{n}\right) ^{\frac{1}{%
\beta -1}}  \label{k><}
\end{equation}
\end{lemma}

\begin{proof}
The statement follows from $\left( TD\right) $ easily$.$
\end{proof}

\subsection{The diagonal upper estimate}

For the proof of Theorem \ref{tslUE1} and \ref{tsPMV} our entry point is
Theorem \ref{tLDUE+}.

\begin{corollary}
\label{cdue} Assume that $\left( \Gamma ,\mu \right) $ satisfies $%
(p_{0}),\left( VD\right) ,\left( TD\right) $ and $\left( E\right) ,$ then
for the function $F$ defined in $\left( \ref{fdef}\right) $
\begin{equation*}
\left( MV\right) \Leftrightarrow \left( \ref{duef}\right) \Leftrightarrow
\left( UE_{F}\right)
\end{equation*}%
and
\begin{equation*}
\left( \ref{PMV}\right) \Leftrightarrow \ \left( MV\right) \Leftrightarrow
\left( \ref{g01}\right) \Leftrightarrow \left( \ref{MVG}\right)
\end{equation*}%
holds as well.
\end{corollary}

\begin{proof}
The statement is immediate from Theorem \ref{tLDUE+} since $\left( TD\right)
+\left( E\right) \Longrightarrow \left( TC\right) $
\end{proof}

The next step is to insert $\left( \ref{UBG}\right) $ into the set of the
equivalent conditions. \

Before we start the proof we give the next statement which is immediate
consequence of \ \ Proposition \ref{tLDLExy}. \

\begin{proposition}
\label{tDUExy}For any weighted graph $(\Gamma ,\mu )$ if we assume $\left(
E\right) $ then
\begin{equation}
P_{2n}(x,x)\geq \frac{c\mu (x)}{V(x,f(2n))},  \label{DLE}
\end{equation}%
furthermore
\begin{equation}
\mathbb{P}(T_{x,R}<n)\leq C\exp \left[ -ck(n,R)\right] ,  \label{kszi}
\end{equation}%
where $k$ is the maximal integer $1\leq k\leq R\leq n$ satisfying $(\ref%
{kdef})$ and $F$ is defined again by $\left( \ref{fdef}\right) .$
\end{proposition}

The next step is to show $\left( \ref{UBG}\right) \Longleftrightarrow \left(
UE_{F}\right) $. \ This is done via $\left( MV\right) $.

\begin{theorem}
\label{tDUEeq}Let us assume that $\left( \Gamma ,\mu \right) $ satisfies $%
\left( p_{0}\right) ,$ $\left( VD\right) ,$ $\left( TD\right) $ \ and $%
\left( E\right) $ then the following statements are equivalent

\begin{enumerate}
\item For a fixed $B=B\left( x,R\right) ,$ $y\in B,d=d\left( x,y\right) $
the upper bound for the Green kernel $\left( \ref{UBG}\right) $ holds:
\begin{equation*}
g^{B}\left( y,x\right) \leq C\sum_{i=F\left( d\right) }^{F\left( R\right) }%
\frac{1}{V(x,f\left( i\right) )},
\end{equation*}

\item for all $u\geq 0$ on $\overline{B}\left( x,R\right) $ harmonic
function in $B\left( x.R\right) ,$ the mean value inequality $\left(
MV\right) $ holds
\begin{equation*}
u\left( x\right) \leq \frac{C}{V(x,R)}\sum_{z\in B\left( x,R\right) }u\left(
z\right) \mu \left( z\right) ,
\end{equation*}

\item the upper estimate $\left( UE_{F}\right) $ holds
\begin{equation*}
p_{n}\left( x,y\right) \leq \frac{C}{V\left( x,f\left( n\right) \right) }%
\exp \left[ -ck\left( n,d\right) \right] .
\end{equation*}%
\newline
\newline
\end{enumerate}
\end{theorem}

\begin{proof}
The combination of Corollary \ref{cdue} and Theorem \ref{tldue-ue} verifies $%
\left( MV\right) \Leftrightarrow \left( UE_{F}\right) $. The implication $%
\left( UE_{F}\right) \Longrightarrow \left( \ref{UBG}\right) $ can be shown
as follows. \ Let us assume $\left( p_{0}\right) $,$\left( VD\right) $,$%
\left( TD\right) $,$\left( E\right) $ and $\left( UE_{F}\right) $. \ We can
start from the definition of the local Green kernel for $B=B\left(
x,R\right) $,$d:=d\left( x,y\right) >0,d<R$ and denote $n=F\left( d\right)
<m=E\left( x,R\right) $%
\begin{equation}
g^{B}\left( y,x\right) =\sum_{i=1}^{n-1}p_{i}^{B}\left( y,x\right)
+\sum_{i=n}^{m-1}p_{i}^{B}\left( y,x\right) +\sum_{i=m}^{\infty
}p_{i}^{B}\left( y,x\right) =:S_{1}+S_{2}+S_{3}  \notag
\end{equation}%
\begin{align*}
S_{3}& =\sum_{j=0}^{\infty }\sum_{z\in B}\frac{1}{\mu \left( x\right) }%
P_{j}^{B}\left( y,z\right) P_{m}^{B}\left( z,x\right) \leq
\sum_{j=0}^{\infty }\sum_{z\in B}P_{j}^{B}\left( y,z\right) \max_{z\in
B}p_{m}^{B}\left( z,x\right) \\
& \leq E_{y}\left( x,R\right) \max_{z\in B}\frac{C}{\sqrt{V(z,f(m))V(x,f(m))}%
}
\end{align*}%
and using $\left( VD\right) ,\left( TC\right) $ and $d\left( x,y\right)
,d\left( x,z\right) <R<f\left( m\right) $ we conclude to
\begin{equation*}
\sum_{i=m}^{\infty }p_{i}^{B}\left( y,x\right) \leq C\frac{E\left(
x,R\right) }{V\left( x,R\right) }.
\end{equation*}%
The first term can be estimates as follows using $\left( UE_{F}\right) :$
\begin{equation*}
S_{1}\leq \sum_{i=1}^{n}\frac{C}{V(y,f\left( i\right) )}\exp \left(
-ck\left( i,d\right) \right) ,
\end{equation*}%
using Lemma \ref{k>} it can be bounded, denoting $a=\log _{A_{E}}\left(
d\right) $
\begin{align*}
& \leq C\frac{E\left( x,d\right) }{V\left( x,d\right) }\sum_{i=1}^{n-1}\frac{%
V\left( x,d\right) }{V(y,f\left( i\right) )}\frac{1}{E\left( x,d\right) }%
\exp \left( -c\left( \frac{E\left( x,d\right) }{i}\right) ^{\frac{1}{\beta -1%
}}\right) \\
& \leq C\frac{E\left( x,R\right) }{V\left( x,R\right) }\sum_{j=1}^{a}D_{V}%
\left( \frac{d}{f\left( F\left( \frac{d}{A_{F}^{j}}\right) \right) }\right)
^{\alpha }\frac{2^{-j+1}F(d)}{E\left( x,d\right) }\exp \left( -c\left(
2^{j}\right) ^{\frac{1}{\beta -1}}\right) \\
& \leq C\frac{E\left( x,d\right) }{V\left( x,d\right) }\sum_{j=1}^{a}\left(
\frac{A_{F}^{\alpha }}{2}\right) ^{j}\exp \left( -c\left( 2^{j}\right) ^{%
\frac{1}{\beta -1}}\right)
\end{align*}%
and it is clear that the sum is bounded by a constant independent of $d$ and
$n$\ which results that
\begin{equation*}
S_{1}=\sum_{i=1}^{n-1}p_{i}^{B}\left( y,x\right) \leq C\frac{E\left(
x,d\right) }{V\left( x,d\right) }.
\end{equation*}%
The estimate of the middle term is straightforward from $\left(
UE_{F}\right) ;$%
\begin{equation*}
S_{2}=\sum_{i=n}^{m-1}p_{i}^{B}\left( y,x\right) \leq \sum_{i=n}^{m}\frac{C}{%
V\left( x,f\left( i\right) \right) }.
\end{equation*}%
\ Finally a trivial estimate shows that \ \ \ \ \
\begin{equation}
S_{3}\leq C\frac{E\left( x,R\right) }{V\left( x,R\right) }\leq C\frac{%
F\left( R\right) }{V\left( x,f(F\left( R\right) )\right) }\leq
C\sum_{i=F\left( d\right) }^{F\left( R\right) }\frac{1}{V(x,f\left( i\right)
)},  \label{mid1}
\end{equation}%
\ \ \ \ \
\begin{equation}
S_{1}\leq C\frac{E\left( x,d\right) }{V\left( x,d\right) }\leq C\frac{%
F\left( d\right) }{V\left( x,f(C^{\prime }F\left( d\right) )\right) }\leq
\sum_{i=F\left( d\right) }^{C^{\prime }F\left( d\right) }\frac{1}{%
V(x,f\left( i\right) )}  \label{nid2}
\end{equation}%
which results
\begin{equation}
g^{B}\left( y,w\right) =S_{1}+S_{2}+S_{3}\leq C\sum_{i=F\left( d\right)
}^{C^{\prime }F\left( R\right) }\frac{1}{V(x,f\left( i\right) )}.
\label{mid3}
\end{equation}%
The\ \ next step is \ to show $\left( \ref{UBG}\right) \Longrightarrow
\left( \ref{g01}\right) $%
\begin{equation*}
g^{B}\left( y,x\right) \leq C\sum_{i=F\left( d\right) }^{C^{\prime }F\left(
R\right) }\frac{1}{V(x,f\left( i\right) )}\leq C\frac{C^{\prime }F\left(
R\right) -F\left( d\right) }{V\left( x,f\left( F\left( d\right) \right)
\right) }\leq C\frac{E\left( x,R\right) }{V\left( x,d\right) }
\end{equation*}%
where in the last step the doubling property of $V$ and $F$ was used. \ We
have seen in Corollary \ref{cdue} that $\left( \ref{g01}\right) $ implies $%
\left( \ref{duef}\right) $ and it implies $\left( UE_{F}\right) $, \ hence
we have shown that $\left( \ref{UBG}\right) \Longrightarrow \left(
UE_{F}\right) $.
\end{proof}

\begin{proof}[Proof of Theorem \protect\ref{tDUE+}]
The result follows from Corollary \ref{cdue} and Theorem \ref{tldue-ue} and %
\ref{tDUEeq}.
\end{proof}

\begin{proof}[Proof of Theorem \protect\ref{tsPMVF}]
The proof is evident from Theorem \ref{tDUE+} and Remark \ref{sPMVF}.
\end{proof}

\subsection{The two-sided estimate}

\label{stwoside}In this section we prove Theorem \ref{tmain}. First we
collect several consequences of the elliptic Harnack inequality which enable
us to apply Theorem \ref{tDUE+}, particularly to deduce $\left(
UE_{F}\right) .$

As we indicated the parabolic and elliptic Harnack inequalities play
important\ role in the study of two-sided bound of the heat kernel. \ Here
we give their formal definitions.

\begin{definition}
The weighted graph $\left( \Gamma ,\mu \right) $ satisfies the ($F-$%
parabolic or simply$)$ parabolic Harnack inequality if \ the following
condition holds. For a given profile $\mathcal{C}=\{c_{1},c_{2},c_{3},c_{4},%
\eta \}$, $0<c_{1}<c_{2}<c_{3}<c_{4},0<\eta <1,$ set of constants, there is
a $C_{H}(\mathcal{C})>0$ constant such that for any solution $u\geq 0$ of
the heat equation
\begin{equation*}
Pu_{n}=u_{n+1}
\end{equation*}%
on $\mathcal{U}=[k,k+F(c_{4}R)]\times B(x,R)$ for $k,R\in \mathbb{N}$ the
following is true. On the smaller cylinders defined by
\begin{equation*}
\mathcal{U}^{-}=[k+F(c_{1}R),k+F(c_{2}R)]\times B(x,\eta R)\text{ and }%
\mathcal{U}^{+}=[k+F(c_{3}R),k+F(c_{4}R)]\times B(x,\eta R)
\end{equation*}%
and taking $(n_{-},x_{-})\in \mathcal{U}^{-},(n_{+},x_{+})\in \mathcal{U}%
^{+},d(x_{-},x_{+})\leq n_{+}-n_{-}$ the inequality
\begin{equation*}
u(n_{-},x_{-})\leq C_{H}\widetilde{u}(n_{+},x_{+})
\end{equation*}%
holds, where $\widetilde{u}_{n}=u_{n}+u_{n+1}$ short notation was used. Let
us remark that $C_{H}$ depends on the constants (including $c_{i},\eta
,D_{V},D_{E},A_{V}$) involved.
\end{definition}

It is standard knowledge that if the (classical) parabolic Harnack
inequality holds for a given profile then it is true for arbitrary profile.
\ We have shown\ in \cite{TD} Subsection 7.1 that the same holds in the
general case if $F$ is proper.

\begin{definition}
The weighted graph $\left( \Gamma ,\mu \right) $ satisfies the elliptic
Harnack inequality $(\mathbf{H})$ if there is a $C>0$ such that for all $%
x\in \Gamma $ and $R>0$ and for all $u\geq 0$ harmonic functions on $B(x,2R)$
the following inequality holds
\begin{equation}
\max_{B(x,R)}u\leq C\min_{B(x,R)}u.  \label{H}
\end{equation}
\end{definition}

The elliptic Harnack inequality is a direct consequence of the $F$-parabolic
one as it is true for the classical case.

The main result of this section is the following, which implies Theorem \ref%
{tmain}.\label{conttwo}

\begin{theorem}
\label{tmain+}If a weighted graph $(\Gamma ,\mu )$ satisfies $\left(
p_{0}\right) $ then the following statements are equivalent.

\begin{enumerate}
\item the F-parabolic Harnack inequality hold for a very\ proper $F$,

\item $(UE_{F})$ and $(LE_{F})$ hold for a very proper $F$

\item $(VD),\left( \ref{rrvA}\right) $ and $\left( H\right) $ hold,

\item $(VD),(E)$ and $\left( H\right) $ hold.
\end{enumerate}
\end{theorem}

\subsubsection{The Einstein relation}

The following five observations are taken from \cite{ter}

\begin{proposition}
\label{prd}If $\left( p_{0}\right) ,\left( VD\right) $ and $\left( H\right) $
holds then the resistance has the doubling properties: there are $%
C,C^{\prime }>1$ such that
\begin{equation}
\frac{\rho (x,R,4R)}{\rho (x,R,2R)}\leq C  \label{rd1}
\end{equation}%
and%
\begin{equation}
\frac{\rho (x,R,4R)}{\rho (x,2R,4R)}\leq C^{\prime }.  \label{rd2}
\end{equation}
\end{proposition}

For the proof see \cite{ter}. The next corollary is trivial consequence of
Proposition \ref{prd}.

\begin{corollary}
\label{crvD}If $\left( p_{0}\right) ,\left( VD\right) $ and $\left( H\right)
$ holds then there is a constant $C>1$ such that
\begin{equation}
\frac{\rho (x,2R,4R)v(x,2R,4R)}{\rho (x,R,2R)v(x,R,2R)}\leq C.
\end{equation}
\end{corollary}

\begin{theorem}
\label{tslER}If for $\left( \Gamma ,\mu \right) $ conditions $\left(
p_{0}\right) ,$ $(VD),(H)$ and $(E)$ hold then
\begin{equation*}
E(x,2R)\simeq \rho (x,R,2R)v(x,R,2R).
\end{equation*}
\end{theorem}

\begin{theorem}
\label{tslERrv}If for the weighted graph $\left( \Gamma ,\mu \right) $ the
conditions $\left( p_{0}\right) ,$ $(VD),\left( H\right) $ and $\left( \ref%
{rrvA}\right) ,$ which is
\begin{equation*}
\rho v(x,R,2R)v(x,R,2R)\simeq \rho v(y,R,2R)v(y,R,2R).
\end{equation*}%
hold then
\begin{equation*}
E(x,2R)\simeq \rho (x,R,2R)v(x,R,2R).
\end{equation*}
\end{theorem}

\begin{proposition}
If $\left( \Gamma ,\mu \right) $ satisfies $\left( p_{0}\right) ,\left(
VD\right) ,\left( H\right) $ and $\left( E\right) $ ( or $\left( \ref{rrvA}%
\right) $ ) then the the function $F$%
\begin{equation*}
F\left( R\right) =\inf_{x\in \Gamma }\rho (x,R,2R)v(x,R,2R)
\end{equation*}%
is proper furthermore the strong ant-doubling property holds. The latter
means that there are $\ B_{F}>A_{F}>1$\ such that
\begin{equation}
F(A_{F}R)\geq B_{F}F(R)  \label{sadF}
\end{equation}%
for all $R>0.$ In short, under the conditions $F$ if very proper.
\end{proposition}

\begin{proof}
From Corollary \ref{cFproper} we know that $F$ is proper and $\left( \ref%
{sadF}\right) $ is shown in \cite{ter} under the conditions.
\end{proof}

\begin{proof}[Proof of Theorem \protect\ref{tmain+}]
The implication $4.\Longrightarrow 3.$ \ is given in Theorem \ref{tslER} and
$3.\Longrightarrow 4$ in Theorem \ref{tslERrv}, $3.\Longrightarrow 2.$ needs
the implication%
\begin{equation*}
\left( p_{0}\right) +\left( VD\right) +\left( TD\right) +\left( H\right)
+(E)\Longrightarrow \left( \ref{duef}\right) ,(UE_{F})
\end{equation*}%
which follows from Theorem \ref{tDUE+} since $\left( H\right)
\Longrightarrow \left( MV\right) .$ The proof of the lower estimate works as
in \cite{TD}. \ The return route $2.\Longrightarrow 1.\Longrightarrow 3.$
also has been shown in \cite[Theorem 2.22]{TD}. \ The only minor
modification is that the condition of annulus resistance doubling ( (2.6)
there ) follows from the doubling property of $F$ and $\rho v$ by Corollary %
\ref{crvD}.
\end{proof}

\section{Example}

\label{sexample}In this section we describe in details of the example of the
stretched Vicsek tree mentioned in the introduction. We show that it
satisfies the conditions of Theorem \ref{tLDUE} and \ref{tLDUE+}.

Let $G_{i}$ is the subgraph of the Vicsek tree (c.f.
\cite{GT2}\ ) which contains the root
$\ z_{0}$ and has diameter $D_{i}=23^{i}$. Let us denote by $z_{i}$ the
vertices on the infinite path, $d\left( z_{0},z_{i}\right) =D_{i}$. \ Denote
$G_{i}^{\prime }=G_{i}\backslash G_{i-1}\cup \left\{ z_{i-1}\right\} $ for $%
i>0,$ the annulus defined by $G$-s.

The new graph is defined by stretching the Vicsek tree as follows. \
Consider the subgraphs $G_{i}^{\prime }$ and replace all the edges of them
by a path of length $i+1$. Denote the new subgraph by $A_{i},$ the new
blocks by $\Gamma _{i}=\cup _{j=0}^{i}A_{i},$ then the new graphs is $\Gamma
=\cup _{j=0}^{\infty }A_{j}$. We denote by $z_{i}$ the cut point between $%
A_{i}$ \ and $A_{i-1}$ again.\ For $x\neq y,x\sim y$ let $\mu _{x,y}=1$.

One can see that neither the volume nor the mean exit time grows
polynomially on $\Gamma $ and both are not uniform on it. On the other hand $%
\Gamma $\ is a tree and the resistance grows asymptotically linearly on it.
We show that $\left( VD\right) $ and $\left( TC\right) $ holds on $\Gamma $
furthermore the elliptic Harnack inequality holds.

Let us recognize some straightforward relations first
\begin{eqnarray}
d\left( z_{0},z_{n}\right) &=&d\left( z_{0},z_{n-1}\right) +2n3^{n}<n3^{n+1}
\label{ln} \\
&<&\left( n+2\right) 3^{n+1},
\end{eqnarray}%
\begin{equation*}
\mu \left( \Gamma _{n}\right) =C\left( 4+\sum_{i=1}^{n}2\left( i+1\right)
4^{i}\right) \simeq n4^{n}\simeq \mu \left( A_{n}\right) ,
\end{equation*}%
\begin{equation*}
\rho \left( \left\{ x\right\} ,B\left( x,R\right) ^{c}\right) \simeq \rho
\left( x,R,2R\right) \simeq R,
\end{equation*}%
\begin{equation*}
E\left( x,R\right) \leq CRV\left( x,R\right) .
\end{equation*}

\begin{lemma}
The tree $\Gamma $ \ satisfies $\left( VD\right) $.
\end{lemma}

\begin{proof}
\bigskip Denote $L_{i}=d\left( z_{0},z_{i}\right) ,d_{i}=\frac{1}{2}\left(
L_{i}-L_{i-1}\right) $ and recognize that $L_{n-1}\simeq L_{n}\simeq d_{n}$.
Let us consider a ball $B\left( x,2R\right) $ and an $N>0$ such that $x\in
A_{N}$ \ and $k$:
\begin{equation*}
d_{k-1}\leq R<d_{k}.
\end{equation*}%
First we assume that the ball is large relative to the position \ of the
centre, which means that it captures basically the large scale property of
the graphs.

Case 1. $k\geq N$.

For convenience we introduce a notation. \ Denote $\Omega _{n}$ \ one of the
blocks of $A_{k}$ of diameter $d_{k}$. \ There is a block $\Omega _{k-2}$
which contains $x$. \ It is clear that $\Omega _{k-2}\subset B\left(
x,R\right) $ and\
\begin{equation*}
V\left( x,R\right) \geq \mu \left( \Omega _{k-2}\right) \simeq \mu \left(
\Gamma _{k+1}\right) .
\end{equation*}%
\ On the other hand $R<L_{k}$ which results that $B\left( x,2R\right)
\subset \Gamma _{k+1}$ and from $\mu \left( \Gamma _{k+1}\right) \simeq \mu
\left( \Omega _{k-2}\right) $ $\left( VD\right) $ follows. \

Case 1. $k<N$.

Now we have to separate sub-cases. \ Again let us fix that $x\in \Omega _{N}$%
. Denote $d=d\left( x,z_{N-1}\right) $. If $x$ \ is not in the central block
of $A_{N}$ then, the $B\left( x,R\right) \subset A_{N}\cup A_{N+1}$ and
since these parts of the graph contain only paths of length of $N+1$ \ or $%
N+2$ volume doubling follows from the fact that it \ holds for the original
Vicsek tree. The same applies if $x$ is in the central block but $B\left(
x,2R\right) \subset A_{N}$. Finally if $B\left( x,2R\right) \cap \Gamma
_{N-1}\neq \varnothing $ \ then $R\geq 2d_{N-1}>d_{N-1}$ \ which means by
the definition of $k$ that $k=N-1,B\left( x,R\right) \supset \Omega _{N-1}$\
\ and on the other hand $B\left( x,2R\right) \subset \Gamma _{N+1}$ \ which
again gives $\left( VD\right) $.
\end{proof}

The elliptic Harnack inequality follows as in \cite{GT2} from the fact that
the Green functions are nearly radial. The linear resistance growth implies
that
\begin{equation*}
\rho \left( \left\{ x\right\} ,B^{c}\left( x,2R\right) \right) \simeq \rho
\left( x,R,2R\right) \simeq R.
\end{equation*}%
\ Let us also recall that from $\left( VD\right) ,\left( H\right) $ and the
linear resistance growth it follows that%
\begin{eqnarray}
&&cRV\left( x,2R\right)  \label{er3} \\
&\leq &c\rho \left( x,R,2R\right) V\left( x,2R\right)  \notag \\
&\leq &E\left( x,2R\right) \\
&\leq &\rho \left( \left\{ x\right\} ,B^{c}\left( x,2R\right) \right)
V\left( x,2R\right) \leq CRV\left( x,2R\right) .  \notag
\end{eqnarray}

The conditions $\left( TC\right) $ follows from $\left( VD\right) $ and $%
\left( \ref{er3}\right) $.

Let us remark that the mean value inequality is implied by the Harnack
inequality and consequently the conditions of Theorem \ref{tLDUE} and \ref%
{tLDUE+} are satisfied.

\section{List of the main conditions}

\begin{eqnarray*}
&&%
\begin{array}{lllll}
\text{shortcut} &  & \text{equation} &  & \text{name} \\
\left( p_{0}\right)  &  & \left( \ref{p0}\right)  &  & \text{controlled
weights condition} \\
\left( VD\right)  &  & \left( \ref{vdef}\right)  &  & \text{volume doubling
property} \\
\left( TC\right)  &  & \left( \ref{TC}\right)  &  & \text{time comparison
principle} \\
\left( TD\right)  &  & \left( \ref{TD}\right)  &  & \text{time doubling
property} \\
\left( MV\right)  &  & \left( \ref{MV}\right)  &  & \text{mean value
inequality} \\
\left( DUE\right)  &  & \left( \ref{LDLE}\right)  &  & \text{diagonal upper
estimate} \\
\left( UE\right)  &  & \left( \ref{UE1}\right)  &  & \text{upper estimate}
\\
\left( E\right)  &  & \left( \ref{E}\right)  &  & \text{uniform mean exit
time } \\
\left( H\right)  &  & \left( \ref{H}\right)  &  & \text{elliptic Harnack
inequality}%
\end{array}
\\
&&%
\begin{array}{lllll}
\left( UE_{F}\right) \text{ \ } &  & \left( \ref{UEF}\right) \text{ \ \ } &
& \text{upper estimate w.r.t. }F\text{ \ \ } \\
\left( LE_{F}\right)  &  & \left( \ref{LEF}\right)  &  & \text{lower
estimate w.r.t. }F%
\end{array}%
\end{eqnarray*}

\end{document}